\documentclass[twoside]{amsart}
\usepackage[centertags]{amsmath}

\usepackage{amsfonts}
\usepackage{amssymb}
\usepackage{amsthm}
\usepackage{eucal}
\usepackage[cmtip,all]{xy}
\usepackage{graphicx}
\newtheorem{thm}[equation]{Theorem}

\newtheorem{prop}[equation]{Proposition}
\theoremstyle{definition}

\theoremstyle{remark}
\newtheorem{rem}[equation]{Remark}

\newcommand{\abs}[1]{\left\vert#1\right\vert}

\newcommand{\eps}{\varepsilon}
\newcommand{\To}{\longrightarrow}

\newcommand{\C}[1]{\mathcal{#1}}

\newcommand{\spec}{\mathcal{S}p}

\newcommand{\EZDIAG}[5]{\xymatrix@C+=2.5cm{*+[r]{#1}
\ar@(u,l)_(0.62){\displaystyle #5}[]
\ar@<.5ex>^-{#3}[r]&\ar@<.5ex>^-{#4}[l]#2}}

\newcommand{\td}[1]{\langle #1\rangle}

\def\r{\rightarrow} % flecha -->
\def\hom{\operatorname{Hom}}

\def\ho{\operatorname{Ho}}

\def\m{\mathfrak{m}}

\def\st{\stackrel} % abreviatura de \stackrel

\def\Z{\mathbb{Z}}

\def\S{\Sigma}

\begin{document}

\title{Triangulated categories without models}%

\author{Fernando Muro}
\address{Universitat de Barcelona, Departament d'$\grave{\text{A}}$lgebra i Geometria, Gran Via de les Corts Catalanes, 585, 08007 Barcelona, Spain}
\email{fmuro@ub.edu}

\author{Stefan Schwede}
\address{Mathematisches Institut, Universit\"at Bonn, Beringstr.~1, 53115 Bonn, Germany}
\email{schwede@math.uni-bonn.de}

\author{Neil Strickland}
\address{Department of Pure Mathematics, University of Sheffield,
Hicks Building, Hounsfield Road, Sheffield S3 7RH, UK}
\email{n.p.strickland@sheffield.ac.uk}

\thanks{The first author was partially supported
by the Spanish Ministry of Education and Science under MEC-FEDER grants
MTM2004-01865 and MTM2004-03629, the postdoctoral fellowship EX2004-0616, and a Juan de la Cierva research contract.}

% ----------------------------------------------------------------
\begin{abstract}
We exhibit examples of triangulated categories which are neither 
the stable category of a Frobenius category nor a 
full triangulated subcategory of the homotopy category of 
a stable model category. Even more drastically,
our examples do not admit any non-trivial exact functors to or from
these algebraic respectively topological triangulated categories.
\end{abstract} \maketitle
% ----------------------------------------------------------------

\subsection*{Introduction}

Triangulated categories are fundamental tools in both algebra and topology. In algebra they often arise as the stable category of a Frobenius category (\cite[4.4]{shc}, \cite[IV.3 Exercise 8]{mha}). In topology they usually appear as a full triangulated subcategory of the homotopy category of a Quillen stable model category~\cite[7.1]{hmc}. The triangulated categories which belong, up to exact equivalence, to one of these two families will be termed \emph{algebraic} and \emph{topological}, respectively. We borrow this terminology from \cite[3.6]{odgc} and \cite{avttc}. Algebraic triangulated categories are generally also topological, but there are many well-known examples of topological triangulated categories which are not algebraic.

In the present paper we exhibit examples 
of triangulated categories which are neither algebraic nor topological.
As far as we know, these are the first examples of this kind.
Even worse (or better, depending on the perspective), 
our examples do not even admit non-trivial exact functors
to or from algebraic or topological triangulated categories.
In that sense, the new examples are completely orthogonal
to previously known triangulated categories.

Let $(R,\m)$ be a commutative local ring with $\m=(2)\neq 0$ and $\m^2=0$.  Examples of this kind of rings are $R=\Z/4$, or more generally $R=W_2(k)$ the $2$-typical Witt 
vectors of length $2$ over a perfect field $k$ of characteristic $2$. There are also examples which do not arise as Witt vectors, for instance the localization of the polynomial ring $\Z/4[t]$ at the prime ideal $(2)$. We denote by $\C{F}(R)$ the category of finitely generated free $R$-modules. 

\begin{thm}\label{rara}
The category $\C{F}(R)$ has a unique structure of a triangulated category 
with identity translation functor and such that the diagram
$$ R\st{2}\To R\st{2}\To R\st{2}\To R $$
is an exact triangle.
\end{thm}

%The number~2 is special here, since there do not exist, 
%in any triangulated category, exact triangles as the one above 
%with the number~2 replaced by an odd integer, 
%compare Remark~\ref{rem-no odd exotic}.

Given an object $X$ in an {\em algebraic} triangulated category $\C{T}$ 
and an exact triangle
$$A\st{2\cdot 1_A}\To A\To C\To\S A,$$
the equation $2\cdot 1_C=0$ holds, compare \cite[3.6]{odgc} and \cite{avttc}. 
Since the ring $R$ satisfies $2\cdot 1_R\neq 0$,
the triangulation of the category $\C{F}(R)$ is not algebraic. 
We cannot rule out the possibility of a topological model 
for $\C{F}(R)$ as easily:
the classical example of $A=S$ the sphere spectrum 
in the stable homotopy category shows that the morphism
$2\cdot 1_C$ can be nonzero in this more general context. 

Nevertheless, $\C{F}(R)$ is not topological either,
which follows from Theorem~\ref{raro}.
Here we call an exact functor between triangulated categories {\em trivial} 
if it takes every object to a zero object.

\begin{thm}\label{raro}
Every exact functor from $\C{F}(R)$ to a topological triangulated category
is trivial.
Every exact functor from  a topological triangulated category 
to $\C{F}(R)$ is trivial.
\end{thm}

\subsection*{Acknowledgements}

We are grateful to Bernhard Keller for helpful conversations on the results 
of this paper, and to Amnon Neeman, who suggested the 
possibility of constructing a triangulated structure on $\C{F}(\Z/4)$ 
by using Heller's theory~\cite{shc}.
% Finally, we have profitted from
%an unpublished note by Neil Strickland in which he shows
%that for $k$ a field of characteristic $2$ the category 
%$\C{F}(k[t]/t^2)$ is triangulated.
%An important difference is that Strickland's example is exact equivalent 
%to the category of compact objects in the derived category 
%of a differential graded ring, and thus is algebraic and topological, 
%see \cite[7.4]{otoc}. 

In the original version of this note the first author alone
constructed the triangulation of the category $\C{F}(\Z/4)$ and proved that
it does not admit any model. The second author joined the project later
by providing a simpler and more general proof 
that the triangulation is not topological.
The third author's contribution was an old preprint on the example
considered in Remark~\ref{rem-eps}, which provided some guidance for
the other results.

\subsection*{The triangulated categories}\label{ttc}

Let $\C{T}$ be an additive category and let $\S\colon
\C{T}\st{\sim}\r\C{T}$ be a self-equivalence that we call \emph{translation functor}. A \emph{candidate triangle} $(f,i,q)$ in $(\C{T},\Sigma)$ is a diagram 
\begin{equation}\label{tri}
A\st{f}\To B\st{i}\To C\st{q}\To \S A,
\end{equation}
where $if$, $qi$, and
$(\S f)q$ are zero morphisms. 
A \emph{morphism} of candidate triangles $(\alpha,\beta,\gamma)\colon (f,i,q)\r(f',i',q')$ is a commutative diagram
\begin{equation*}%\label{candmor}
\xymatrix{A\ar[r]^f\ar[d]_{\alpha}&B\ar[r]^i\ar[d]_{\beta}&C\ar[r]^q\ar[d]^{\gamma}&\Sigma A\ar[d]^{\Sigma \alpha}\\
A'\ar[r]_{f'}&B'\ar[r]_{i'}&C'\ar[r]_{q'}&\Sigma A'} 
\end{equation*}
The category of candidate triangles is additive. The \emph{mapping cone} of the morphism $(\alpha,\beta,\gamma)$ is the candidate triangle
\begin{equation*}%\label{tri}
\xymatrix@C=35pt{B\oplus A'\ar[r]^-{\begin{scriptsize}
\left(\hspace{-5pt}\begin{array}{cc}
-i\hspace{-7pt}&0\\
\beta\hspace{-7pt}&f'
\end{array}\hspace{-6pt}\right)\end{scriptsize}
}& C\oplus B'\ar[r]^-{\begin{scriptsize}
\left(\hspace{-5pt}\begin{array}{cc}
-q\hspace{-7pt}&0\\
\gamma\hspace{-7pt}&i'
\end{array}\hspace{-6pt}\right)
              \end{scriptsize}
}& \Sigma A\oplus C'\ar[r]^-{\begin{scriptsize}
\left(\hspace{-5pt}\begin{array}{cc}
-\Sigma f\hspace{-7pt}&0\\
\Sigma \alpha\hspace{-7pt}&q'
\end{array}\hspace{-6pt}\right)\end{scriptsize}}& \Sigma B\oplus\Sigma A'.}
\end{equation*}
A \emph{homotopy} $(\Theta,\Phi,\Psi)$ from $(\alpha,\beta,\gamma)$ to $(\alpha',\beta',\gamma')$ is given by morphisms
\begin{equation*}%\label{homoto}
\xymatrix{A\ar[r]^f\ar@<-.5ex>[d]_-{\alpha}\ar@<.5ex>[d]^-{\alpha'}&
B\ar[r]^i\ar@<-.5ex>[d]_-{\beta}\ar@<.5ex>[d]^-{\beta'}\ar[ld]_<(.3)\Theta&
C\ar[r]^q\ar@<-.5ex>[d]_-{\gamma}\ar@<.5ex>[d]^-{\gamma'}\ar[ld]_<(.3)\Phi&
\Sigma A\ar@<-.5ex>[d]_-{\Sigma\alpha}\ar@<.5ex>[d]^-{\Sigma\alpha'}\ar[ld]_<(.3)\Psi\\
A'\ar[r]_{f'}&B'\ar[r]_{i'}&C'\ar[r]_{q'}&\Sigma A'} 
\end{equation*}
such that
$$\begin{array}{rclcrclcrcl}
\beta'-\beta\!\!\!&=\!\!\!&\Phi i+f'\Theta,&&\gamma'-\gamma\!\!\!&=\!\!\!&\Psi q+i'\Phi,&&\Sigma(\alpha'-\alpha)\!\!\!&=\!\!\!&\Sigma(\Theta f)+q'\Psi.
\end{array}$$
We say in this case that the morphisms are \emph{homotopic}.
The mapping cones of two homotopic morphisms are isomorphic.
A \emph{contractible triangle} is a candidate triangle such that the identity is homotopic to the zero morphism. A homotopy $(\Theta,\Phi,\Psi)$ from $0$ to $1$ is called a \emph{contracting homotopy}. Any morphism from or to a contractible triangle is always homotopic to zero.

A \emph{triangulated category} is a pair $(\C{T},\Sigma)$ as above together with a collection of candidate triangles, called \emph{distinguished} or \emph{exact triangles}, satisfying the following properties. The family of exact triangles is closed under isomorphisms. The candidate triangle
\begin{equation}\label{trivial}
A\st{1}\To A\To 0\To \Sigma A,
\end{equation}
is exact. Any morphism $f\colon A\r B$ in $\C{T}$ can be extended to an exact triangle like (\ref{tri}). A candidate triangle (\ref{tri}) is exact if and only if its \emph{translate}
\begin{equation*}%\label{tra}
B\st{-i}\To C\st{-q}\To \S A \st{-\Sigma f}\To \S B,
\end{equation*}
is exact. Any commutative diagram
\begin{equation*}%\label{fill}
\xymatrix{A\ar[r]^f\ar[d]_{\alpha}&B\ar[r]^i\ar[d]_{\beta}&C\ar[r]^q&\Sigma A\ar[d]^{\Sigma \alpha}\\
A'\ar[r]_{f'}&B'\ar[r]_{i'}&C'\ar[r]_{q'}&\Sigma A'} 
\end{equation*}
whose rows are exact triangles can be extended to a morphism whose mapping cone is also exact. This non-standard set of axioms for triangulated categories is equivalent to the classical one, see \cite{triang}, and works better for the purposes of this paper.

Now we are ready to prove Theorem \ref{rara}.

\begin{proof}[Proof of Theorem \ref{rara}]
Given an object $X$ in $\C{F}(R)$ we consider the candidate triangle $X_2$ defined as
\begin{equation}\label{ut}
X\st{2}\To X\st{2}\To X\st{2}\To X.
\end{equation}

We are going to prove that the category $\C{F}(R)$ has a triangulated category structure with identity translation functor where the exact triangles are the candidate triangles isomorphic to the direct sum of a contractible triangle and a candidate triangle of the form~\eqref{ut}.

The family of exact triangles is closed under isomorphisms by definition. The candidate triangle~\eqref{trivial} is contractible, and hence exact. 
The ring $R$ is a quotient of a discrete valuation ring with maximal ideal 
generated by $2$, see~\cite[Corollary 3]{sclr};
therefore any morphism $f\colon A\r B$ in $\C{F}(R)$ 
can be decomposed up to isomorphism as
$$f=\left(\begin{array}{ccc}
1&0&0\\
0&2&0\\
0&0&0
\end{array}\right)\colon A= W\oplus X\oplus Y\To W\oplus X\oplus Z=B.$$
Then $f$ is extended by the direct sum of~\eqref{ut} 
and the contractible triangle
\begin{equation*}%\label{tri}
\xymatrix@C=35pt{W\oplus Y\ar[r]^-{\begin{scriptsize}
\left(\hspace{-5pt}\begin{array}{cc}
1\hspace{-7pt}&0\\
0\hspace{-7pt}&0
\end{array}\hspace{-5pt}\right)\end{scriptsize}
}& W\oplus Z\ar[r]^-{\begin{scriptsize}
\left(\hspace{-5pt}\begin{array}{cc}
0\hspace{-7pt}&0\\
0\hspace{-7pt}&1
\end{array}\hspace{-5pt}\right)
              \end{scriptsize}
}& Y\oplus Z\ar[r]^-{\begin{scriptsize}
\left(\hspace{-5pt}\begin{array}{cc}
0\hspace{-7pt}&0\\
1\hspace{-7pt}&0
\end{array}\hspace{-5pt}\right)\end{scriptsize}}& W\oplus Y.}
\end{equation*}
The translate of a contractible triangle is also contractible, 
and the triangle~\eqref{ut} is invariant under translation.
%\begin{equation*}
%\xymatrix{X\ar[r]^{-t}\ar[d]_{1}^\cong&X\ar[r]^{-vt %}\ar[d]^\cong_{1}&X\ar[r]^{-t}\ar[d]^\cong_{v^{-1}}&X\ar[d]_{1}^\cong\\
%X\ar[r]_{t}&X\ar[r]_{t}&X\ar[r]_{vt}&X} 
%\end{equation*}
%Here we use that $\cha k=2$ and $\m^2=0$, so $\m$ is a $k$-vector space and $2t=0$. 
This proves that the translate of an exact triangle is exact. Translating a candidate triangle six times yields the original one, therefore if a candidate triangle has an exact translate then the original candidate triangle is also exact.

We say that a candidate triangle
$A\st{f}\r B\st{i}\r C\st{q}\r A$
is a \emph{quasi-exact triangle} if
\begin{equation*}
A\st{f}\To B\st{i}\To C\st{q}\To A\st{f}\To B,
\end{equation*}
is an exact sequence of $R$-modules. The exact triangles are all quasi-exact.

Now we are going to show that any diagram of candidate triangles
\begin{equation}\label{fill2}
\xymatrix{A\ar[r]^f\ar[d]_{\alpha}&B\ar[r]^i\ar[d]_{\beta}&C\ar[r]^q&A\ar[d]^{\alpha}\\
A'\ar[r]_{f'}&B'\ar[r]_{i'}&C'\ar[r]_{q'}&A'} 
\end{equation}
with exact rows can be completed to a morphism with exact mapping cone.

Suppose that the
upper row in \eqref{fill2} is contractible and the lower row is quasi-exact. Since $f'\alpha=\beta f$ then $f'\alpha q=0$; since $C$ is projective, there exists $\gamma'\colon C\r C'$ such that $q' \gamma'=\alpha q$. Let $(\Theta,\Phi,\Psi)$ be a contracting homotopy for the upper row. Then $\gamma=\gamma'+(i'\beta-\gamma'i)\Phi$ completes (\ref{fill2}) to a morphism of candidate triangles.

If the upper row in (\ref{fill2}) is quasi-exact and the lower row is contractible then (\ref{fill2}) can also be completed to a morphism. This can be shown directly, but it also follows from the previous case since we have a duality functor 
$$\hom_R(-,R)\colon\C{F}(R)\st{\sim}\To\C{F}(R)^{op},$$
which preserves contractible triangles and quasi-exact triangles. Here we use that $R$ is injective as an $R$-module, see \cite[Example 3.12]{lmr}.

If the upper and the lower rows in (\ref{fill2}) are $X_2$ and $Y_2$, respectively, then $\gamma=\beta+2\delta$ extends (\ref{fill2}) to a morphism of candidate triangles for any $\delta\colon X\r Y$. 

%Suppose now that we have a commutative square with exact rows.
%\begin{equation}\label{fill3}
%\xymatrix{X\ar[r]^t\ar[d]_{\alpha}&X\ar[r]^t\ar[d]_{\beta}&X\ar[r]^{vt}&X\ar[d]_{\alpha}\\
%Y\ar[r]_{t}&Y\ar[r]_{t}&Y\ar[r]_{vt}&Y} 
%\end{equation}
%Since $u,v\in R$ both project to $z\in k$ we have that $ut=vt$, therefore 

This proves that any diagram like (\ref{fill2}) with exact rows can be completed to a morphism $\varphi=(\alpha,\beta,\gamma)$. Now we have to check that the completion can be done in such a way that the mapping cone is exact. Suppose that the upper and the lower rows are $X_2\oplus T$ and $Y_2\oplus T'$, respectively, with $T$ and $T'$ contractible. The morphism $\varphi$ is given by a matrix of candidate triangle morphisms
$$\varphi=\left(\begin{array}{cc}
\varphi_{11}&\varphi_{12}\\
\varphi_{21}&\varphi_{22}
\end{array}\right)\colon X_2\oplus T\To Y_2\oplus T',$$
where $\varphi_{ij}=(\alpha_{ij},\beta_{ij},\gamma_{ij})$.
Here $\varphi_{12}$, $\varphi_{21}$ and $\varphi_{22}$ are homotopic to $0$ since either the source or the target is contractible, therefore the mapping cone of $\varphi$ is isomorphic to the mapping cone of 
$$\psi=\left(\begin{array}{cc}
\varphi_{11}&0\\
0&0
\end{array}\right)\colon X_2\oplus T\To Y_2\oplus T',$$
which is the direct sum of the mapping cone of $\varphi_{11}$ and two contractible triangles, $T'$ and the translate of $T$. 

We can suppose that
$$\alpha_{11}=\left(\begin{array}{ccc}
1&0&0\\
0&2&0\\
0&0&0
\end{array}\right)\colon X= L\oplus M\oplus N\To L\oplus M\oplus P=Y.$$
Moreover, as we have seen above we can take $\gamma_{11}=\beta_{11}+2\delta$ for 
\begin{eqnarray*}
\delta&=&\left(\begin{array}{ccc}
0&0&0\\
0&1&0\\
0&0&0
\end{array}\right)\colon X= L\oplus M\oplus N\To L\oplus M\oplus P=Y.
\end{eqnarray*} 
We have $2\beta_{11}=2\alpha_{11}$, therefore $\beta_{11}=\alpha_{11}+2\Phi$ for some $\Phi\colon X\r Y$. Now we observe that $(\delta,\Phi,0)$ is a homotopy from $\varphi_{11}$ to $\zeta=(\alpha_{11}+2\delta,\alpha_{11}+2\delta,\alpha_{11}+2\delta)$, so the mapping cone of $\varphi_{11}$ is isomorphic to the mapping cone of $\zeta$. 

The mapping cone of $\zeta$ is clearly the direct sum of five candidate triangles, namely $M_2$, $N_2$, $M_2$ (once again), $P_2$, and the mapping cone of the identity $1\colon L_2\r L_2$, which is contractible. Therefore the mapping cone of $\zeta$ is exact, and also the mapping cone of $\varphi_{11}$, $\psi$ and $\varphi$.

It remains to show the uniqueness claim in Theorem~\ref{rara}.
In any triangulation, all contractible candidate triangles are
exact \cite[1.3.8]{triang}. The triangle $X_2$ is a finite direct sum of copies of $R_2$.
Hence every triangulation of $(\C{F}(R),\text{Id})$ which contains $R_2$
contains all the exact triangles which we considered above.
Two triangulations with the same translation functor necessarily agree
if one class of triangles is contained in the other,
so there is only one triangulation in which $R_2$ is exact.
This completes the proof.
\end{proof}

\begin{rem}\label{rem-exactness}
 The exact triangles in $\C{F}(R)$ can be characterized more
 intrinsically as follows.  Let $T$ be a quasi-exact triangle, which
 we can regard as a $\Z/3$-graded chain complex of free $R$-modules
 with $H_*(T)=0$.  As $T$ is free we have a short exact sequence
 $$2T\hookrightarrow T\st{2}\twoheadrightarrow 2T,$$ and the resulting long exact sequence in homology
 reduces to an isomorphism $\sigma\colon H_*(2T)\to H_{*-1}(2T)$.  As the
 grading is $3$-periodic we can regard $\sigma^3$ as an automorphism
 of $H_*(2T)$.  We claim that $T$ is exact if and only if
 $\sigma^3=1$.  One direction is straightforward: if $T$ is
 contractible then $H_*(2T)=0$, and if $T=X_2$ then $H_i(2T)=2X$ for
 all $i$ and $\sigma$ is the identity.  The converse is more fiddly
 and we will not go through the details.  It would be nice to give a
 proof of Theorem~\ref{rara} based directly on this definition of
 exactness, but we do not know how to do so.
\end{rem}

\begin{rem}\label{rem-eps} Let $k$ be a field of characteristic~2.
The same arguments as in the proof of Theorem~\ref{rara}
show that the category 
$\C{F}(k[\eps]/\eps^2)$ of finitely generated free modules
over the algebra $k[\eps]/\eps^2$ of dual numbers admits a triangulation
with the identity translation functor and such that the diagram
\begin{equation*}%\label{eps triangle}
k[\eps]/\eps^2 \st{\eps}\To k[\eps]/\eps^2 \st{\eps}\To k[\eps]/\eps^2 
\st{\eps}\To k[\eps]/\eps^2  
\end{equation*}
is an exact triangle. However, this triangulated category
is both algebraic and topological, and hence, 
from our current perspective, less interesting.

Indeed $\C{F}(k[\eps]/\eps^2)$ is an algebraic and topological triangulated category for any field $k$. The translation functor $\Sigma=\tau^*$ is the restriction of scalars along the $k$-algebra automorphism $\tau\colon k[\eps]/\eps^2\r k[\eps]/\eps^2$ with $\tau(\eps)=-\eps$, and
\begin{equation}\label{eps triangle2}
k[\eps]/\eps^2 \st{\eps}\To k[\eps]/\eps^2 \st{\eps}\To k[\eps]/\eps^2 
\st{\eps}\To \tau^* k[\eps]/\eps^2  
\end{equation}
is an exact triangle. An algebraic model for this triangulated category was obtained by Keller in \cite{otoc}. Keller's model is a differerential graded (dg) $k$-category. Here we exhibit an alternative model, which is a dg $k$-algebra $A$ such that 
$\C{F}(k[\eps]/\eps^2)$ is exact equivalent
to the category of compact objects 
in the derived category $\C{D}(A)$ of dg (right) $A$-modules. This shows that $\C{F}(k[\eps]/\eps^2)$ is both algebraic and topological.

Let $A= k\td{a,u,v,v^{-1}}/I$  be the free graded $k$-algebra generated by
$a$, $u$, $v$ and $v^{-1}$ in degrees $\abs{a}=\abs{u}=0$ and $\abs{v}=-1$ 
modulo the two-sided homogeneous ideal $I$ generated by 
$$a^2,\;\; au+ua+1,\;\; av+va\;\text{ and }\; uv+vu.$$ 
The differential $d:A\to A$ is determined by $$d(a)=u^2v,\;\; d(u)=0,\;\; d(v)=0,$$ and the Leibniz rule.
The ungraded algebra $H_0(A)$ is isomorphic to the dual numbers
$k[\eps]/\eps^2$, where $\eps=[u]$ is the homology class of the cycle~$u$. 
The graded algebra $H_*(A)$ is determined by this isomorphism since $[v]$ is a unit in degree $-1$ such that $\eps\cdot [v]+[v]\cdot\eps =0$.

We claim that the $0$-dimensional homology functor 
$$ H_0 \colon  \C{D}^c(A)  \To  \C{F}(k[\eps]/\eps^2) $$
is an equivalence of categories, where the left hand side is 
the full subcategory of those dg $A$-modules 
whose $H_0$ is finitely generated over $k[\eps]/\eps^2$.

Let $M$ be any dg $A$-module and let $[x]\in H_0(M)$ be a homology
class with $[x]\cdot\eps=0$. 
We choose a representing cycle $x$ and an element $y$ with $d(y)=xu$; 
then the element $z=yuv-xa$ is a cycle with $x=zu-d(ya)$, 
so $[x]=[z]\cdot \eps$ in homology. 
So every homology class which is annihilated by $\eps$ 
is also divisible by $\eps$, 
which proves that $H_0(M)$ is a free $k[\eps]/\eps^2$-module. Moreover, the translation functor in $\C{D}(A)$ is the usual shift of complexes $M\mapsto M[1]$ and the natural isomorphism $$\tau^*H_0(M)\cong H_0(M[1])=H_{-1}(M)$$
is given by $[x]\mapsto [xv]$.

The universal case of this is $M=A\{x,y\}$, the free graded right module over the underlying graded algebra of $A$ with $\abs{x}=0$ and $\abs{y}=1$. We can endow $M$ with a dg $A$-module structure with
$d(x)=0$ and $d(y)=xu$, so that $M$ is just the mapping cone of the chain
map $A\st{u}\to A$.  
The cycle $z=yuv-xa\in M$ gives a quasiisomorphism $A\to M$. 
Using this, we obtain an exact triangle 
$$ A  \st{u}\To A \st{u}\To A \st{uv}\To A[1]$$
in $\C{D}^c(A)$ which maps to the exact triangle~\eqref{eps triangle2}.
The rest of the proof that $H_0$ is an exact equivalence from 
$\C{D}^c(A)$ to $\C{F}(k[\eps]/\eps^2)$ is relatively straightforward,
and we omit it.
\end{rem}

We still owe the proof that the triangulated category $\C{F}(R)$ 
does not admit non-trivial exact functors to or from 
a topological triangulated category.
For this purpose we introduce two intrinsic properties that an object $A$ of
a triangulated category may have.

A {\em Hopf map} for an object $A$ is a morphism $\eta:\Sigma A\to A$
which satisfies $2\eta=0$ and such that for some (hence any) exact triangle
\begin{equation}\label{eq:tri_top}
A \st{2}\To A \st{i}\To  C \st{q}\To  \Sigma A
\end{equation}
we have $i\eta q=2\cdot 1_C$. An object which admits a Hopf map will be termed \emph{hopfian}.
We note that the class of hopfian objects is closed
under isomorphism, suspension and desuspension. If $F$
is an exact functor with natural isomorphism $\tau:\Sigma F\cong F\Sigma$
and $\eta:\Sigma A\to A$ a Hopf map for $A$,
then the composite $F(\eta)\tau:\Sigma F(A)\To F(A)$
is a Hopf map for $F(A)$.

We call an object $E$ {\em exotic} if there exists an exact triangle
\begin{equation}  \label{eq:tri_exotic}
E \st{2}\To E \st{2}\To  E \st{h}\To  \Sigma E
\end{equation}
for some morphism $h:E\to\Sigma E$.
We note that the class of exotic objects is closed
under isomorphism, suspension and desuspension.
Every exact functor takes exotic objects to exotic objects.
Every object of the triangulated category $\C{F}(R)$ of Theorem~\ref{rara}
is exotic.

We remark without proof that the morphism $h$ which makes \eqref{eq:tri_exotic} exact is
unique and natural for morphisms between exotic objects. We show %in Remark
%12 that 
below that $h$ is of the form $h=2\psi$ for an isomorphism $\psi:E\to\S E$.

\begin{rem}\label{rem-no odd exotic}   
The integer~2 plays a special role in the definition
of exotic objects, which ultimately comes from the sign which
arises in the rotation of a triangle. In more detail, suppose that there is
an exact triangle
\begin{equation}  \label{eq:tri_exotic_n}
E \st{n}\To E \st{n}\To  E \st{h}\To  \Sigma E 
\end{equation}
for some integer $n$. We claim that if $E$ is nonzero, 
then $n\equiv 2\mod 4$ %is congruent to 2 modulo 4 
and $4\cdot 1_E=0$, so that the
triangle~\eqref{eq:tri_exotic_n} equals the `exotic' 
triangle~\eqref{eq:tri_exotic} with $n=2$.
Indeed, we can find a morphism 
$\psi:E\to \Sigma E$ which makes the diagram
$$ \xymatrix{
E\ar[r]^-n\ar@{=}[d]& E\ar[r]^-n\ar@{=}[d]& E\ar[r]^-h \ar@{..>}[d]^\psi&
\S E \ar@{=}[d]\\
E\ar[r]_-n & E\ar[r]_-h& \S E\ar[r]_-{-n}& \S E} $$
commute, and $\psi$ is an isomorphism. We have $n\psi=h=-n\psi$
which gives $2n\psi=0$. Since $\psi$ is an isomorphism, this forces
$2n\cdot 1_E=0$. Exactness of~\eqref{eq:tri_exotic_n} lets us choose
a morphism $f:E\to E$ with $2\cdot 1_E=n\cdot f$.
But then $4\cdot 1_E= n^2 f^2=0$. So if $n$ is divisible by~4, then $E=0$.
If $n$ is odd, then $E$ is anhihilated by~4 and the odd number~$n^2$,
so also $E=0$.
\end{rem}

Hopf maps are incompatible with the property of being exotic
in the sense that these two classes of objects are orthogonal.

\begin{prop}\label{orthogonal} 
Let $\C{T}$ be a triangulated category, $A$ a hopfian object and $E$ an exotic object. Then the morphism 
groups $\C{T}(A,E)$ and $\C{T}(E,A)$ are trivial. 
In particular, every exotic and hopfian object is a zero object.
\end{prop}
\begin{proof} 
Let $\eta:\Sigma A\to A$ be a Hopf map. Given any morphism $f:E\to A$ there exists $g\colon E\r C$ such that $(f,f,g)$ is a morphism from \eqref{eq:tri_exotic} to \eqref{eq:tri_top}, and hence
$if=2g=i\eta q g=i\eta(\Sigma f)h=i\eta(\Sigma f)2\psi=0$. Here we use the notation of Remark \ref{rem-no odd exotic} for $n=2$ and the fact that $2\eta=0$. Moreover, 
\eqref{eq:tri_top} is exact, so $f=2f'$ for some $f'\colon E\r A$. This equation follows for any morphism $f\colon E\r A$, hence $f$ is divisible by any power of $2$, but $4\cdot 1_E=0$, so $f=0$.

The proof of $\C{T}(A,E)=0$ is similar. Alternatively, we can reduce 
this statement to the previous one by observing
that the properties of being exotic and hopfian are self-dual.
In other words, an object $E$ is exotic in a triangulated category $\C{T}$
if and only if $E$ is exotic as an object of the opposite category 
$\C{T}^{op}$ with the opposite triangulation, and similarly for Hopf maps.
\end{proof}

\begin{prop}\label{Hopf maps}
Every object of a topological triangulated category is hopfian.
\end{prop}
\begin{proof}
We can assume that the topological triangulated
category is $\ho\C{M}$ for a stable model category $\C{M}$.
We use that for every object $A$ of $\ho\C{M}$ there exists
an exact functor $F:\ho\spec\to\ho\C{M}$ from the stable homotopy category 
which takes the sphere spectrum to an object isomorphic to $A$.
Here $\spec$ is the category of `sequential spectra' of simplicial sets
with the stable model structure of Bousfield and Friedlander~\cite[Sec.~2]{BF}.
To construct $F$ we let $X$ be a cofibrant-fibrant object 
of the model category $\C{M}$
which is isomorphic to $A$ in the homotopy category $\ho\C{M}$.
The universal property of the model category of spectra~\cite[Thm.~5.1 (1)]{ss}
provides a Quillen adjoint functor pair
$$\xymatrix@C=15mm{ \spec \quad \ar@<.4ex>[r]^-{X\wedge\ } &
\ar@<.4ex>[l]^-{\hom(X,-)} \quad \C{M} }$$
whose left adjoint $X\wedge$ takes the sphere spectrum $S$ to
$X$, up to isomorphism.
The left derived functor of the left Quillen functor 
$X\wedge -:\spec\to\C{M}$ is exact and can serve as the required functor $F$.

Since exact functors preserve Hopf maps it thus suffices 
to treat the `universal example', i.e., 
to exhibit a Hopf map for the sphere spectrum as an object 
of the stable homotopy category.
The stable homotopy class $\eta:\Sigma S\to S$ of the Hopf map
from the 3-sphere to the 2-sphere precisely has this property, hence the name.
In more detail, we have an exact triangle
$$ S \st{2\cdot 1_S}\To  S \st{i}\To  S/2 \st{q}\To  \Sigma S $$
in the stable homotopy category, where $S/2$ is the mod-2 Moore spectrum;
then the morphism $2\cdot 1_{S/2}$ factors as $i\eta q$, and moreover
$2\eta=0$.
\end{proof}

In topological triangulated categories, something a little
stronger than Proposition~\ref{Hopf maps} is true
in that Hopf maps can be chosen {\em naturally} for all objects. 
However, we don't need this and so we omit the details. 
Now we can give the

\begin{proof}[Proof of Theorem~\ref{raro}.]
Every object of the triangulated category $\C{F}(R)$ is exotic and
every object of a topological triangulated category is hopfian.
So an exact functor from one type of triangulated category to
the other hits objects which are both exotic and hopfian.
But such objects are trivial by Proposition~\ref{orthogonal}.
\end{proof}

\begin{rem} The only special thing we use in the proof of 
Theorem~\ref{raro} about {\em topological} triangulated categories 
is that therein every object has a Hopf map. Hopf maps
can also be obtained from other kinds of structure that were proposed by 
different authors in order to `enrich' or `enhance' 
the notion of a triangulated category. So our argument also proves
that the triangulated category $\C{F}(R)$ of Theorem~\ref{rara}
does not admit such kinds of enrichments, and every exact functors 
to or from such enriched triangulated categories is trivial.
For example, if $\C{T}$ is an algebraic triangulated category,
then for some (hence any) exact triangle~\eqref{eq:tri_top} 
we have $2\cdot 1_C=0$; so the zero map is a Hopf map.

Another example of such extra structure is the notion of a 
{\em triangulated derivator}, due to Grothendieck~\cite{der}, and the
closely related notions of a {\em stable homotopy theory}
in the sense of Heller~\cite{ht, shts} or a {\em system of triangulated diagram
categories} in the sense of Franke~\cite{utctc}. 
In each of these settings, the stable homotopy category 
is the underlying category of the free example on one generator 
(the sphere spectrum).
We do not know a precise reference of this fact for triangulated derivators, 
but we refer to~\cite[Cor.~4.19]{puekd}
for the `unstable' (i.e., non-triangulated) analog.
In Franke's setting the universal property is formulated as
Theorem~4 of~\cite{utctc}.
These respective universal properties in the enhanced context provide,
for every object $A$, an exact functor $(\ho\spec)^{\text{cp}}\to\C{T}$
which takes the sphere spectrum $S$ to $A$, up to isomorphism.
This functors sends the classical Hopf map for the sphere spectrum 
to a Hopf map for $A$.

Another kind of structure which underlies many triangulated categories
is that of a {\em stable infinity category} as investigated by Lurie
in~\cite{sic}.
The appropriate universal property of the infinity category of spectra
is established in~\cite[Cor.~17.6]{sic}, so again every object
of the homotopy category of any stable, presentable infinity category
has a Hopf map. 
\end{rem}

\end{document}